\documentclass[english, 12pt]{amsproc}
\usepackage[english]{babel}
\usepackage{a4wide}
\usepackage{amsthm,stackrel}
\usepackage{graphics}
\usepackage{amsfonts, amssymb, amscd, amsmath}
\usepackage{latexsym}
\usepackage[matrix,arrow,curve]{xy}
\usepackage{mathabx}
\usepackage{color}
\usepackage{tikz}
\usetikzlibrary{matrix}

\DeclareMathOperator{\cone}{cone}
\DeclareMathOperator{\curv}{curv}\DeclareMathOperator{\Ann}{Ann}
\DeclareMathOperator{\Vol}{Vol}\DeclareMathOperator{\st}{st}
\DeclareMathOperator{\Td}{Td}\DeclareMathOperator{\NE}{NE}

\newcommand{\Zo}{\mathbb{Z}}
\newcommand{\Ro}{\mathbb{R}}
\newcommand{\Rg}{\mathbb{R}_{\geqslant 0}}
\newcommand{\Co}{\mathbb{C}}

\newcommand{\CP}{\mathbb{C}P}

\newcommand{\intX}{\int_{X_\Delta}}
\newcommand{\ct}{\widetilde{c}}
\newcommand{\dd}{\partial}
\newcommand{\D}{\mathcal{D}}
\newcommand{\F}{\mathcal{F}}

\newcounter{stmcounter}

\theoremstyle{plain}

\newtheorem{thm}[stmcounter]{Theorem}

\newtheorem{prop}[stmcounter]{Proposition}
\newtheorem{lem}[stmcounter]{Lemma}

\theoremstyle{definition}
\newtheorem{defin}[stmcounter]{Definition}

\theoremstyle{remark}

\newtheorem{rem}[stmcounter]{Remark}
\newtheorem{con}[stmcounter]{Construction}

\begin{document}

\title{Toric manifolds over 3-polytopes}

\author{Anton Ayzenberg}
\thanks{This work is supported by the Russian Science Foundation under grant
 14-11-00414}
\address{Steklov Mathematical Institute of Russian Academy of Sciences, Moscow}
\email{ayzenberga@gmail.com}



\begin{abstract}
In this note we gather and review some facts about existence of
toric spaces over 3-dimensional simple polytopes. First, over
every combinatorial 3-polytope there exists a quasitoric manifold.
Second, there exist combinatorial 3-polytopes, that do not
correspond to any smooth projective toric variety. We restate the
proof of the second claim which does not refer to complicated
algebro-geometrical technique. If follows from these results that
any fullerene supports quasitoric manifolds but does not support
smooth projective toric varieties.
\end{abstract}

\maketitle

\section{Introduction}

For a 3-dimensional simple polytope $P$ one can construct a
6-dimensional manifold with the action of the compact torus $T^3$,
whose orbit space is $P$. The topology of this manifold tells a
lot about the combinatorics of the polytope. There exist several
constructions of such manifolds arising in different areas of
mathematics: toric varieties in algebraic geometry and singularity
theory, symplectic toric manifolds in symplectic geometry,
quasitoric manifolds in algebraic topology. Each construction
requires certain properties from the polytope, and these
properties affect the geometrical structure of the resulting
manifold. For example, the construction of a quasitoric manifold
as an identification space \cite{DJ} requires only the
combinatorial type of a polytope, and the resulting manifold is
just a topological manifold. However, if we fix the affine
realization of a polytope, the resulting quasitoric manifold
attains smooth structure, see \cite{BPR}. The construction of a
symplectic toric manifold requires a polytope to be Delzant
\cite{Delz}. There also exist certain conditions on the polytope,
which imply the existence of almost complex, or algebraical
structure on the corresponding manifold.

There are several natural questions. What does the existence of
certain geometrical structure on a toric space say about the
combinatorics of the polytope? The constructions of toric topology
and toric geometry allow to construct a lot of examples of
6-dimensional manifolds. But how large is the set of examples
having certain geometrical structure? In this paper we gather some
known results.

There is a well-known correspondence between toric varieties and
rational fans. Projective toric varieties correspond to normal
fans of convex polytopes and smooth projective toric varieties
correspond to unimodular polytopal fans (i.e. normal fans of
Delzant polytopes). Quasitoric manifolds are the
algebro-topological generalization of smooth projective toric
varieties. A smooth compact manifold $M$ of real dimension $2n$
with the action of half-dimensional compact torus $T^n$ is called
\emph{quasitoric} if
\begin{enumerate}
\item the action is locally standard (i.e. locally modeled by the
standard action of $T^n$ on $\Co^n$ by coordinate-wise rotations);
\item the orbit space (which is a manifold with corners as follows
from pt.1) is isomorphic to some simple polytope $P$.
\end{enumerate}
In this case we say that $M$ is a quasitoric manifold over $P$.
Recall that $n$-dimensional convex polytope is called
\emph{simple} if each of its vertices lies in exactly $n$ facets
(equivalently: each vertex lies in exactly $n$ edges). Among all
convex polytopes only simple polytopes are manifolds with corners.

Every smooth projective toric variety $X$ is a quasitoric
manifold: we can restrict the action of an algebraic torus
$(\Co^\times)^n$ on $X$ to its compact subtorus; and the orbit
space of this action can be identified with the image of the
moment map, which is a simple polytope. However there exist many
quasitoric manifolds which are not toric varieties. The simplest
example is $\CP^2\hash\CP^2$: this is a quasitoric manifold which
is not even algebraical. On the other hand there also exist smooth
non-projective toric varieties which are not quasitoric~\cite{Su}.

In this paper we discuss two basic theorems:

\begin{thm}[\cite{DJ,BP}]\label{thm3cols}
There exists a quasitoric manifold over any 3-dimensional simple
polytope.
\end{thm}

\begin{thm}[\cite{Delaunay}]\label{thmMain}
If there exists a smooth projective toric variety over a simple
3-dimensional polytope $P$, then $P$ has at least one triangular
or quadrangular face.
\end{thm}

The recent interest to these results arose in connection with
fullerenes. Mathematically, a fullerene is a simple 3-dimensional
polytope having only pentagonal and hexagonal faces. Buchstaber
\cite{BuchFuller} suggested to study fullerenes from the
perspective of toric topology: presumably, both fields may benefit
from such interaction.

Two theorems above show that (1) there exist quasitoric manifolds
over fullerenes; (2) there are no smooth projective toric
varieties over fullerenes. Therefore, due to their rigid
geometrical nature, smooth projective toric varieties are not
suited for the study of fullerenes. However, quasitoric manifolds
do: this subject will be discussed in the forthcoming paper of
Buchstaber, Erokhovets, Masuda, Panov, and Park.

We should make a remark that Theorems \ref{thm3cols},\ref{thmMain}
are based on completely different methods. Theorem \ref{thm3cols}
essentially relies on the four color theorem: there are no known
proofs which do not use this result. Theorem \ref{thmMain} was
formulated and proved by Delaunay in \cite{Delaunay}, and its
proof is based on the work of Reid \cite{Reid} concerning Mori's
minimality theory for toric varieties. We restate the proof in
more combinatorial topological terms, without referring to this
algebro-geometrical theory, to make the difference between toric
and quasitoric cases more transparent.

\section{Quasitoric manifolds}

Let $M$ be a quasitoric manifold of dimension $2n$. Its orbit
space under the action of $T^n$ is a simple polytope $P$. Let
$\F_1,\ldots,\F_m$ be the facets of the polytope $P$. Any point
$x$ in the interior of a facet $\F_i$ represents an
$(n-1)$-dimensional orbit of the action. The stabilizer of this
orbit is a 1-dimensional toric subgroup $G_i\subset T^n$. We may
assume that $G_i=\exp(\lambda_{i,1},\ldots,\lambda_{i,n})$, where
$(\lambda_{i,1},\ldots,\lambda_{i,n})\in \Zo^n$ is a primitive
integral vector determined uniquely up to sign. One-dimensional
stabilizer subgroups define the so called \emph{characteristic
function}. Let $[m]=\{1,\ldots,m\}$ be the index set of facets of
$P$. Consider the function $\lambda\colon [m]\to \Zo^n$,
$\lambda\colon i\mapsto (\lambda_{i,1},\ldots,\lambda_{i,n})$.
Actually, the value is determined uniquely up to sign, however we
make a choice of this sign arbitrarily (this corresponds to the
choice of orientation of each stabilizer $G_i$). Since the action
of the torus is locally standard, characteristic function
satisfies the condition:

\begin{equation}\label{eqStarCond}
\mbox{if facets }\F_{i_1},\ldots,\F_{i_n} \mbox{ intersect in a
vertex, then } \lambda(i_1),\ldots,\lambda(i_n) \mbox{ is a basis
of } \Zo^n.
\end{equation}

Therefore, with any quasitoric manifold, one can associate a
characteristic pair $(P,\lambda)$, where $P$ is a simple polytope
and $\lambda$ is a characteristic function.

\begin{con} The construction above can be reverted \cite{DJ}. Given any simple
polytope $P$ with facets $\F_1,\ldots,\F_m$ and a function
$\lambda\colon[m]\to\Zo^n$ satisfying condition
\eqref{eqStarCond}, we can construct a quasitoric manifold
$M_{(P,\lambda)}$ as follows. For each $i\in [m]$ let
$G_i=\exp(\lambda(i))\subset T^n$ be the corresponding circle
subgroup. Take any point $x\in P$; it lies in the interior of some
face $F\in P$. We have $F=\F_{i_1}\cap\cdots\cap\F_{i_k}$. Let
$G_x$ denote the toric subgroup $G_{i_1}\times\cdots\times
G_{i_k}\subset T^n$. Consider the identification space
\[
M_{(P,\lambda)}=(P\times T^n)/\sim
\]
where $(x,t)\sim(x',t')$ whenever the points $x$, $x'$ coincide
and $t't^{-1}\in G_x$. One can check that $M_{(P,\lambda)}$ is a
topological manifold, and there is a locally standard action of
$T^n$, which rotates the second coordinate and gives the orbit
space $P$. A canonical smooth structure on $M_{(P,\lambda)}$ was
constructed in \cite{BPR}. Thus $M_{(P,\lambda)}$ is a quasitoric
manifold.
\end{con}

To prove Theorem \ref{thm3cols} one needs to show that every
simple 3-polytope admits a function $\lambda$, satisfying
condition \eqref{eqStarCond}. This is done by the four colors
theorem.

\begin{proof}[Proof of Theorem \ref{thm3cols}]
Let $c\colon\{\F_1,\ldots,\F_m\}\to\{a,b,c,d\}$ be the coloring of
facets of $P$ in four colors such that adjacent facets have
distinct colors. Let $e_1,e_2,e_3$ be the basis of the lattice
$\Zo^3$. Replace colors by the vectors as follows: $a\mapsto e_1$,
$b\mapsto e_2$, $c\mapsto e_3$, $d\mapsto e_1+e_2+e_3$. This gives
a characteristic function, since every three vectors among
$(e_1,e_2,e_3,e_1+e_2+e_3)$ form a basis of the lattice.
\end{proof}

\section{Toric varieties}

Let $V\cong\Ro^n$ be an oriented real vector space with the fixed
lattice $\Zo^n\cong N\subset V$. Recall that \emph{a fan} in
$\Ro^n$ is a collection of convex cones with apex at the origin
such that the intersection of each two cones of the collection is
a face of both and lies in the collection. The fan is called
\emph{complete} if the union of all cones is the whole space $V$.
The fan is called \emph{rational} if all cones are generated by
rational vectors. The cone is called simplicial (resp. unimodular)
if it is generated by linearly independent vectors of $V$ (resp.
part of a basis of the lattice $N$). The fan is called
\emph{simplicial} (resp. \emph{unimodular}) if all its cones are
such. Every unimodular fan is simplicial.

Let $P$ be a convex polytope in the dual space $V^*$. With any
such polytope one associates \emph{the normal fan}: for each face
$F\subset P$ take the cone spanned by outward normal vectors to
the facets of $P$ containing $F$, and take the collection of these
cones. Normal fan is complete. Normal fan of a simple polytope is
simplicial. Normal fans of polytopes are called polytopal fans.
Note that there exist non-polytopal complete fans.

\begin{defin}
A polytope $P$ is called \emph{Delzant} if its normal fan is
unimodular.
\end{defin}

It follows that every Delzant polytope is simple.

Toric varieties are classified by rational fans. Compact toric
varieties correspond to complete fans. Smooth toric varieties
correspond to unimodular fans. Projective toric varieties
correspond to polytopal fans. Therefore, smooth projective toric
varieties correspond to normal fans of Delzant polytopes (i.e.
polytopal unimodular fans). Theorem \ref{thmMain} can be restated
as follows.

\begin{prop}[\cite{Delaunay}]\label{propDelzantMain}
Let $P$ be a 3-dimensional Delzant polytope. Then $P$ has at least
one triangular or quadrangular face.
\end{prop}

Let $P$ be an $n$-dimensional Delzant polytope with facets
$\F_1,\ldots,\F_m$ and let $\lambda(i)\in\Zo^n$ be the primitive
outward normal vector to $\F_i$. The unimodularity property of the
normal fan of $P$ implies that the function
$\lambda\colon[m]\to\Zo^n$ defined this way satisfies condition
\eqref{eqStarCond}. Therefore each Delzant polytope determines the
characteristic pair $(P,\lambda)$ in a natural way. Smooth
projective toric variety corresponding to the normal fan of $P$ is
equivariantly diffeomorphic to the quasitoric manifold determined
by the pair $(P,\lambda)$, see \cite{DJ}. Due to this observation
smooth projective toric varieties are particular cases of
quasitoric manifolds.

Now we introduce some notation to be used in the following. Let
$\Delta$ denote a complete unimodular fan in $V\cong\Ro^n$ and $m$
be the number of rays in $\Delta$. Let $X_\Delta$ be the smooth
compact toric variety determined by this fan. The underlying
simplicial sphere $K$ of the fan $\Delta$ has $m$ vertices and
dimension $n-1$. Let $\lambda\colon[m]\to N$ be the characteristic
function, that is $\lambda(i)\in N$ is the primitive generator of
$i$-th ray of the fan $\Delta$.

\textbf{Cohomology.}

\begin{thm}[Danilov--Jurkiewicz]\label{thmDanJurk}
$H^*(X_\Delta;\Zo)\cong \Zo[K]/\Theta$, where
\[
\Zo[K]=\Zo[v_1,\ldots,v_m]/(v_{i_1}\cdots v_{i_s}\mid
\{i_1,\ldots,i_s\}\notin K),\qquad \deg v_i=2
\]
is the Stanley--Reisner ring of the sphere $K$, and ideal $\Theta$
is generated by linear forms $\sum_{i\in
[m]}\langle\mu,\lambda(i)\rangle v_i$, for each linear functional
$\mu\colon N\to\Zo$.
\end{thm}

A similar theorem was proved by Davis and Januszkiewicz for
quasitoric manifolds: in this case $K$ is a simplicial sphere dual
to a polytope, and $\lambda$ is a characteristic function. Similar
theorems hold with real coefficients instead of integers.

Let $\intX\colon H^{2n}(X_\Delta;\Zo)\to \Zo$ denote the pairing
with the fundamental class of $X_\Delta$. Consider a subset
$I=\{i_1,\ldots,i_n\}\subset[m]$. We have
\begin{equation}\label{eq0or1}
\intX v_{i_1}\cdots v_{i_n}=\begin{cases} 1, \mbox{ if } I\in K\\
0, \mbox{ otherwise.}\end{cases}
\end{equation}
In the following we also need the description of tangent Chern
classes of $X_\Delta$.

\begin{thm}[\cite{Ehl}]
Under the isomorphism of Theorem \ref{thmDanJurk} the $j$-th Chern
class of the tangent bundle of the manifold $X_\Delta$ corresponds
to the elementary symmetric polynomial in the variables~$v_i$:
\[
c_j(X_\Delta)=\sigma_j(v_1,\ldots,v_m)=\sum_{I\in
K,|I|=j}\prod_{i\in I}v_i\in H^{2j}(X_\Delta;\Zo).
\]
\end{thm}

A completely similar theorem was proved in \cite{DJ} for a
quasitoric manifold after introducing certain stably complex
structure on it.

\textbf{Effective cone.}

The notion of effective cone is one of the essential points in the
proof of Theorem \ref{thmMain}. This notion is defined in
algebraic geometry for arbitrary projective varieties, however we
restrict to the smooth case, where it has a clear geometrical
meaning. This subsection is needed only for the completeness of
the exposition: for toric varieties all necessary notions will be
defined in combinatorial-geometrical manner below.

Let $X$ be an arbitrary smooth K\"{a}hler manifold. Each compact
complex curve $C\subset X$ determines a homology class $[C]\in
H_2(X;\Ro)$, which is called \emph{effective}. The set of all
nonnegative linear combinations of effective classes in
$H_2(X;\Ro)$ is called \emph{the effective cone} of the manifold
$X$:
\[
\NE(X)=\left\{\sum r_i[C_i]\in H_2(X;\Ro)\mid r_i\geqslant
0\right\}.
\]

\begin{prop}
$\NE(X)$ is a strictly convex cone in $H_2(X;\Ro)$.
\end{prop}

\begin{proof}
We need to prove that all effective classes lie in some open
half-space of $H_2(X;\Ro)$. Consider the class of K\"{a}hler form
$\omega\in H^2(X;\Ro)$. For each complex curve $C$ we have
\[
\langle\omega,[C]\rangle=\int_C\omega|_C=\Vol(C)>0.
\]
This means that all effective classes lie in the half-space
\[
\{\alpha\in H_2(X;\Ro)\mid \langle \omega,\alpha\rangle>0\},
\]
which implies the statement.
\end{proof}

\begin{prop}[\cite{Reid}]
Let $X$ be smooth projective toric variety. Then its effective
cone $\NE(X)$ is polyhedral and is generated by the fundamental
classes of torus-invariant 2-spheres (preimages of edges of the
polytope under the projection to the orbit space).
\end{prop}

The generators of the effective cone are called \emph{extremal
cycles}. Note that in general not all edges of the polytope define
extremal cycles: some of them may lie in the cone generated by
others.

\textbf{Effective cone in toric case: combinatorial-geometrical
approach.}

Here we introduce all the necessary notions from the previous
paragraph in combinatorial manner. Let $X_\Delta$ be the smooth
projective variety corresponding to a polytopal fan $\Delta$.

The simplices of $K$ of codimension $1$ as well as the
corresponding cones of $\Delta$ will be called \emph{the walls}.
For each wall $J=\{i_1,\ldots,i_{n-1}\}\in K$ consider the class
$v_J=v_{i_1}\cdots v_{i_{n-1}}\in H^{2n-2}(X;\Ro)$. Note that
$v_J\neq 0$ by obvious reasons. Consider the cone in
$H^{2n-2}(X_\Delta;\Ro)$ generated by the classes $v_J$ for all
walls $J\in K$:
\[
\NE(X_\Delta)=\left\{\sum r_Jv_J\in H^{2n-2}(X_\Delta;\Ro)\mid
r_J\geqslant 0\right\}
\]

\begin{prop}
For each smooth projective toric variety the effective cone
$\NE(X_\Delta)$ is a strictly convex polyhedral cone in
$H^{2n-2}(X_\Delta;\Ro)$.
\end{prop}

\begin{proof}
Let $V_\Delta\in \Ro[c_1,\ldots,c_m]$ be the volume polynomial of
the fan $\Delta$. By definition,
\[
V_\Delta(c_1,\ldots,c_m)=\frac{1}{n!}\intX(c_1v_1+\ldots+c_mv_m)^n.
\]
It is known (see \cite{PKh}), that the values of this polynomial
are the volumes of simple polytopes with the normal fan $\Delta$.
More precisely, let $P=\{x\in V^*\mid \langle
x,\lambda(i)\rangle\leqslant \ct_i\}$ be a simple convex polytope
with the normal fan $\Delta$ (since $X_\Delta$ is projective, at
least one such polytope exists). The numbers $\ct_i$ are called
the support parameters of $P$. Then we have
$\Vol(P)=V_\Delta(\ct_1,\ldots,\ct_m)$. To avoid the mess, we
denote the formal variables of the volume polynomial by $c_i$, and
concrete real numbers substituted in this polynomial are denoted
by $\ct_i$.

Let $\dd_i=\frac{\dd}{\dd c_i}$ be the differential operators,
acting on $\Ro[c_1,\ldots,c_m]$. Let $\D=\Ro[\dd_1,\ldots,\dd_m]$
be the commutative algebra of differential operators with constant
coefficients, and $\Ann V_\Delta=\{D\in\D\mid DV_\Delta=0\}$ be
the annihilating ideal of the polynomial $V_\Delta$. According to
\cite{PKh2,Tim}, we have
\[
\D/\Ann V_\Delta\cong H^*(X_\Delta;\Ro),\qquad
\dd_i\leftrightarrow v_i.
\]
Moreover, the integration map $\intX\colon H^{2n}(X_\Delta;\Ro)\to
\Ro$ coincides with the natural map $(\D/\Ann V_\Delta)_{n}\to
\Ro$, $D\mapsto DV_\Delta$. To prove the statement it is
sufficient to show that the classes
\[
\{\dd_J=\dd_{i_1}\cdots\dd_{i_{n-1}}\in (\D/\Ann
V_\Delta)_{n-1}\mid J=\{i_1,\ldots,i_{n-1}\}\in K\}
\]
lie in one open half-space.

Let $P$ be a convex polytope with the normal fan $\Delta$ and
support parameters $\ct_i$. Consider the element $\dd_c=\sum_{i\in
[m]}\ct_i\dd_i\in (\D/\Ann V_\Delta)_1$. Recall a simple fact: for
each homogeneous polynomial $\Psi\in \Ro[c_1,\ldots,c_m]$ of
degree $k$ there holds
$\frac{1}{k!}\dd_c^k\Psi=\Psi(\ct_1,\ldots,\ct_m)$ (this is a
particular case of Euler's theorem on homogeneous functions).

\begin{lem}
Let $J\in K$ be a wall. Then $\dd_c\dd_JV_\Delta>0$.
\end{lem}

\begin{proof}
Note that $\dd_JV_\Delta$ is a linear polynomial in variables
$c_i$. Therefore, the number $\dd_c\dd_JV_\Delta$ coincides with
the value of the polynomial $\dd_JV_\Delta$ in the point
$\ct=(\ct_1,\ldots,\ct_m)$ by the preceding remark. It is known
that the value of the polynomial $\dd_JV_\Delta$ in the point
$\ct$ coincides up to positive factor with the length of the edge
$F_J\subset P$, dual to the wall $J\in K$ (this was noted by
Timorin in \cite{Tim}, and in \cite{AyM} we proved that the factor
is the volume of the parallelepiped spanned by
$\lambda(i_1),\ldots,\lambda(i_{n-1})$). Thus
$\dd_c\dd_JV_\Delta>0$.
\end{proof}

According to lemma, all classes $\dd_J\in \D/\Ann V_\Delta$ lie in
the half-space $\{D\mid D\dd_cV_\Delta>0\}$ which implies the
statement.
\end{proof}

\begin{defin}
Let $J=\{i_1,\ldots,i_{n-1}\}\in K$ be a wall such that $v_J\in
H^{2n-2}(X_\Delta;\Ro)$ is a generating element of the effective
cone $\NE(X_\Delta)$. Then $J$ is called \emph{an extremal
simplex} and $v_J$ is called \emph{an extremal class}.
\end{defin}

The condition of being extremal can be written as follows. Suppose
an extremal class $v_J$ is expressed as a sum $\nu_1+\nu_2$, where
$\nu_1,\nu_2\in \NE(X_\Delta)$. Then both elements $\nu_1$,
$\nu_2$ lie in the ray generated by $v_J$.

\begin{rem}
This definition agrees with the general theory. The vector spaces
$H^{2n-2}(X_\Delta;\Ro)$ and $H_2(X_\Delta;\Ro)$ can be identified
by Poincare duality, and under this identification the class
$v_J=v_{i_1}\cdots v_{i_{n-1}}$ corresponds to the fundamental
class of torus-invariant 2-sphere obtained as a transversal
intersection of characteristic submanifolds
$X_{i_1},\ldots,X_{i_{n-1}}$ (preimages of facets
$\F_{i_1},\ldots,\F_{i_{n-1}}$ under the projection to the orbit
space).
\end{rem}

\section{Unimodular geometry of fans}

An arbitrary wall $J=\{i_1,\ldots,i_{n-1}\}\in K$ is contained in
exactly two maximal simplices: $I=\{i_1,\ldots,i_{n-1},i\}$ and
$I'=\{i_1,\ldots,i_{n-1},i'\}$. Both sets of vectors
\[
\{\lambda(i_1),\ldots,\lambda(i_{n-1}),\lambda(i)\},\qquad
\{\lambda(i_1),\ldots,\lambda(i_{n-1}),\lambda(i')\}
\]
are the bases of the lattice. Write $\lambda(i')$ in the first
basis:
\[
\lambda(i')=a_1\lambda(i_1)+\ldots+a_{n-1}\lambda(i_{n-1})-\lambda(i).
\]
(Unimodularity condition of the set $\lambda(I')$ guarantees that
the coefficient at $\lambda(i)$ is $\pm1$. The fact that cones at
$I$ and $I'$ lie on opposite sides of the wall $J$ guarantees that
the coefficient at $\lambda(i)$ is exactly $-1$.) In what follows
we assume that vertices $i_1,\ldots,i_{n-1}$ are ordered such that
$\lambda(i_1),\ldots,\lambda(i_{n-1}),\lambda(i)$ is a positive
basis of the lattice, while, respectively,
$\lambda(i_1),\ldots,\lambda(i_{n-1}),\lambda(i')$ is a negative
basis.

\begin{defin}
The number
\[\curv(J)=2-a_1-\ldots-a_{n-1}\in \Zo\]
is called \emph{the unimodular curvature} of the wall $J$.
\end{defin}

The underlying simplicial complex $K$ of a fan $\Delta$ may be
realized in $V\cong \Ro^n$ as a star-shaped sphere as follows: let
us send the vertex $i$ to the point $\lambda(i)\in V$ and continue
the map on each simplex by linearity. We denote the image of this
map by $\st(K)$; it is a piecewise linear sphere in $V$ winding
around the origin.

We say that $\st(K)$ is concave (resp. convex, resp. flat) at the
wall $J$, if the affine hyperplane, through the points
$\lambda(i_1),\ldots,\lambda(i_{n-1}),\lambda(i)$ separates
$\lambda(i')$ from the origin (resp. does not separate, resp.
contains $\lambda(i')$).

\begin{lem}\label{lemCurvProps}
A unimodular curvature and parameters $a_1,\ldots,a_{n-1}$,
defined above satisfy the following properties.

\begin{enumerate}
\item $a_s=\det(\lambda(i_1),\ldots,\lambda(i_{s-1}),\lambda(i'),
\lambda(i_{s+1}), \ldots,\lambda(i_{n-1}),\lambda(i))$.

\item The star-shaped sphere $\st(K)$ is convex (resp. flat, resp. concave)
at a wall $J$ if and only if $\curv(J)>0$ (resp. $\curv(J)=0$,
resp. $\curv(J)<0$).

\item In complete simplicial fan there exists a wall of positive
curvature.

\item $\intX v_Jv_{i_s}=-a_s$, \quad $\curv(J)=\intX v_J(\sum_{t\in[m]}v_t)$.
\end{enumerate}
\end{lem}

\begin{proof}
(1) Take exterior product of the relation
\begin{equation}\label{eqAsRelation}
\lambda(i)+\lambda(i')=\sum_{t=1}^{n-1}a_t\lambda(i_t)
\end{equation}
with the exterior form $\lambda(i_1)\wedge\cdots\wedge
\widehat{\lambda(i_s)}\wedge\cdots\wedge\lambda(i_{n-1})\wedge\lambda(i)$.
The result is the desired relation.

(2) The convexity of the star-shaped sphere $\st(K)$ at a wall $J$
depends on spatial relationship between the affine line through
the points $\lambda(i),\lambda(i')$ and the codimension 2 affine
subspace through the points
$\lambda(i_1),\ldots,\lambda(i_{n-1})$, that is on the sign of the
determinant
\begin{multline*}
\det(\lambda(i_1)-\lambda(i'),\ldots,\lambda(i_{n-1})-\lambda(i'),\lambda(i)-\lambda(i'))
=\\=\det(\lambda(i_1),\ldots,\lambda(i_{n-1}),\lambda(i))-
\det(\lambda(i_1),\ldots,\lambda(i_{n-1}),\lambda(i'))-\\-
\sum_{s=1}^{n-1} \det(\lambda(i_1),\ldots,
\stackrel[s]{}{\lambda(i')},\ldots
\lambda(i_{n-1}),\lambda(i))=1-(-1)-\sum_{s=1}^{n-1}a_s=\curv(J).
\end{multline*}

(3) If the curvature of any wall is non-positive, then the
star-shaped sphere $\st(K)$ could not wind around the origin.

(4) We express the class $v_{i_s}$ through the classes $v_j,
j\notin J$, by using linear relations in the cohomology ring.
Consider the linear functional $\mu$ on the space $V$ such that
$\langle \mu,\lambda(i)\rangle=0$ and
\[
\langle\mu,\lambda(i_t)\rangle=\begin{cases}0,\mbox{ если }t\neq
s,\\ 1,\mbox{ если }t=s.
\end{cases}
\]
Applying $\mu$ to relation \eqref{eqAsRelation}, we get
$\langle\mu,\lambda(i')\rangle=a_s$. It follows that there is a
linear relation
$v_{i_s}+a_sv_{i'}+\sum_{j\notin\{i_1,\ldots,i_{n-1},i,i'\}}C_jv_j$
in the cohomology ring. Let us multiply this relation by $v_J$.
Since $J$ forms a simplex only with vertices $i,i'$,
Stanley--Reisner relations imply
\[
\intX v_Jv_{i_s}=\intX -a_sv_Jv_{i'}=-a_s.
\]
The formula for the curvature easily follows:
\[
\intX \left(v_J\cdot\sum\nolimits_{t\in[m]}v_t\right)=\intX
v_Jv_i+\intX v_Jv_{i'}+\sum_{s=1}^{n-1}\intX
v_Jv_{i_s}=2-\sum_{s=1}^{n-1}a_s=\curv(J).
\]
\end{proof}

There is an interesting corollary (which also gives an alternative
proof of point (3) of the previous Lemma in dimension $n=3$).

\begin{prop}[Unimodular Gauss--Bonnet theorem]
Let $\Delta$ be a unimodular simplicial fan of dimension $3$. Then
the sum of curvatures of all its walls equals $24$.
\end{prop}

\begin{proof}
It follows from the previous lemma, that
\[
\sum_{J\in K,|J|=2}\curv(J)=\intX(\sum_{J\in
K,|J|=2}v_J)(\sum_{t\in[m]}v_t)=\intX
c_2(X_\Delta)c_1(X_\Delta)=c_{1,2}(X_\Delta).
\]
It is known that for stably complex manifolds of real dimension
$6$ the Chern number $c_{1,2}(X_\Delta)$ coincides with
$24\Td(X_\Delta)$. Todd genus of a smooth compact toric variety
equals $1$, and the statement follows.
\end{proof}

Now we prove Theorem \ref{thmMain}. Let $\Delta$ be the normal fan
of a Delzant polytope $P$. The walls of this fan are simply the
edges of the 2-dimensional triangulated sphere $K$.

\begin{proof}[Proof of Theorem \ref{thmMain}]
According to Lemma \ref{lemCurvProps}(3), there exists a wall
$\tilde{J}\in K$ of positive curvature. On the other hand, Lemma
\ref{lemCurvProps}(4) implies that the curvature of the wall
$\tilde{J}$ coincides with the value of the linear functional
$H^4(X_\Delta;\Ro)\to\Ro$, $u\mapsto\intX(u\cdot c_1(X_\Delta))$
on the effective class $v_{\tilde{J}}$. Since a linear functional
takes positive value on some element of the effective cone, this
functional should take positive value on some generator of this
cone. Therefore, there exists an extremal wall $J=\{i_1,i_2\}\in
K$ having positive curvature.

Let $a_1,a_2$ be the parameters of the wall $J$, defined earlier.
Since $\curv(J)=2-a_1-a_2>0$ and the numbers $a_1,a_2$ are
integral, we have either $a_1\leqslant 0$, or $a_2\leqslant 0$.
Assume $a_1\leqslant 0$. Consider two cases:

(1) $a_1<0$. Let us prove that in this case $i_2$ is contained in
exactly three maximal cones. As before, let $I=\{i,i_1,i_2\}$,
$I'=\{i',i_1,i_2\}$ be the maximal simplices containing the wall
$J$. Suppose that apart from the vertices $i_1,i,i'$ the vertex
$i_2$ is connected to the vertices $k_1,\ldots,k_p$, $p\geqslant
1$ (we assume that the neighbors of the vertex $i_2$ are
cyclically ordered as $i_1,i',k_1,\ldots,k_p,i$, see
Fig.\ref{pictRays}).

\begin{figure}[h]
\begin{center}
\includegraphics[scale=0.2]{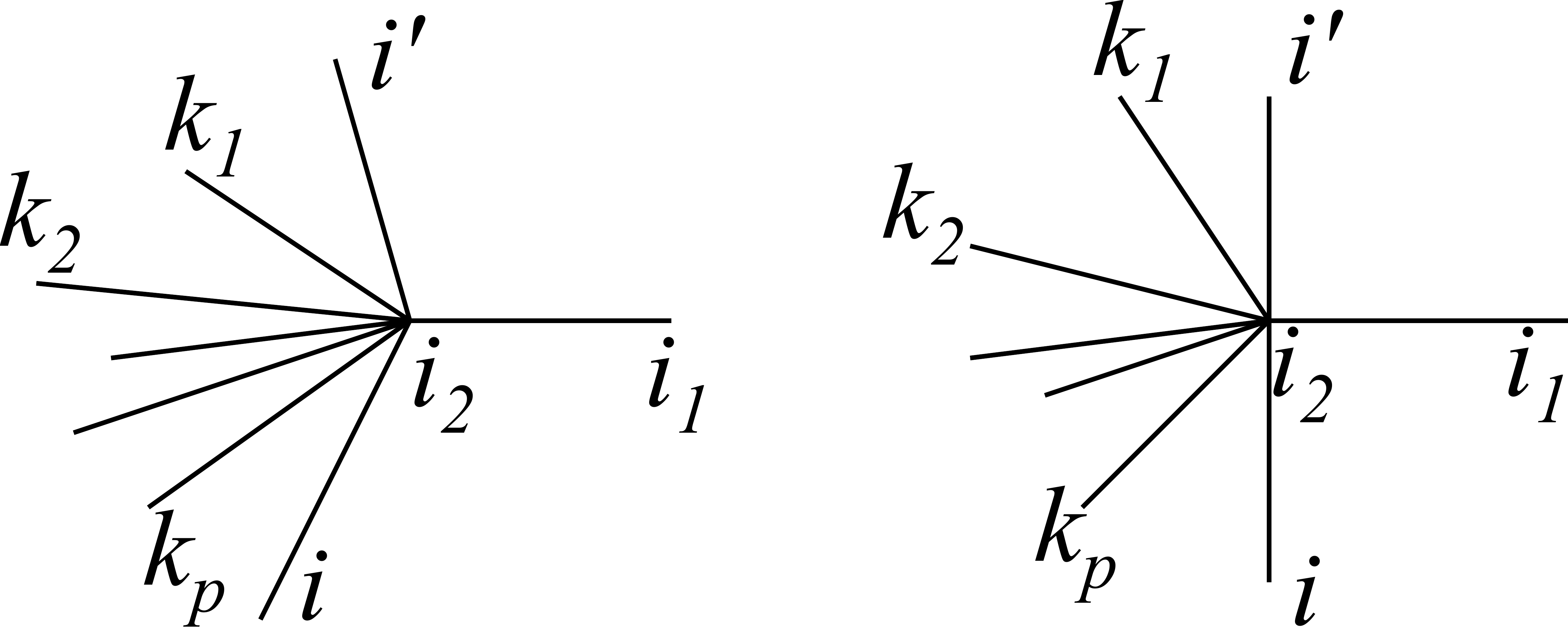}
\end{center}
\caption{The vicinity of the ray $\Rg\lambda(i_2)$ in the first
and the second cases. The ray $\Rg\lambda(i_2)$ points to the
reader.}\label{pictRays}
\end{figure}

According to Lemma \ref{lemCurvProps}(1),
$a_1=\det(\lambda(i'),\lambda(i_2),\lambda(i))<0$. This means,
that the sum of dihedral angles of the cones
$C(I)=\cone(\lambda(i_1),\lambda(i_2),\lambda(i))$ and
$C(I')=\cone(\lambda(i_1),\lambda(i_2),\lambda(i'))$ at the edge
$\Rg\lambda(i_2)$ exceeds one straight angle (see left part of
Fig.\ref{pictRays}). There exists a 2-plane $\Pi$ which contains
the ray $\Rg\lambda(i_2)$ and separates $\lambda(i_1)$ from the
vectors
\begin{equation}\label{eqListVecs}
\lambda(i),\lambda(i'),\lambda(k_1),\ldots,\lambda(k_p).
\end{equation}
Let $\mu$ be the linear functional on $\Ro^3$, annihilating the
plane $\Pi$ and taking value $1$ on the vector $\lambda(i_1)$. By
construction, $\mu$ takes negative values on all vectors from the
list \eqref{eqListVecs}. In the cohomology ring we have a linear
relation
\[
v_{i_1}=\sum_{t\in \{i,i',k_1,\ldots,k_p\}}C_tv_t+\sum_{t\notin
\{i_1,i_2,i,i',k_1,\ldots,k_p\}}D_tv_t,
\]
in which all coefficients $C_t$ are positive. Multiplying this
relation by $v_{i_2}$, we get
\[
v_J=v_{i_1}v_{i_2}=\sum_{t\in
\{i,i',k_1,\ldots,k_p\}}C_tv_tv_{i_2}
\]
(the part of expression, having coefficients $D_t$ vanishes by
Stanley--Reisner relations). Therefore, the class $v_J$ is
expressed as a positive linear combination of the classes
$v_tv_{i_2}$, $t\in \{i,i',k_1,\ldots,k_p\}$. Since $v_J$ was
chosen to be extremal, each of the classes $v_tv_{i_2}$ is
proportional to the class $v_J$. Since all these classes are
nonzero, they are all proportional to each other. This leads to
contradiction. Indeed, according to relation \eqref{eq0or1} we
have $(v_{i'}v_{i_2})v_{k_1}\neq 0$ since $\{i',i_2,k_1\}\in K$,
but $(v_{i_1}v_{i_2})v_{k_1}=0$ since $\{i_1,i_2,k_1\}\notin K$.

(2) $a_1=0$. We prove that in this case the vertex $i_2$ is
contained in four maximal cones. The proof is similar to the
previous case. Assume the contrary: let the vertex $i_2$ have the
neighbors $i,i_1,i',k_1,\ldots,k_p$, $p\geqslant 2$, written in
the cyclic order.

According to Lemma \ref{lemCurvProps}(1), the condition $a_1=0$
implies that the vectors $\lambda(i_2),\lambda(i),\lambda(i')$
belong to a single 2-plane, say $\Pi$. Let $\mu$ be the linear
functional annihilating $\Pi$ and taking value $1$ on the vector
$\lambda(i_1)$. Consequently, $\mu$ takes strictly negative values
on the vectors $\lambda(k_1),\ldots,\lambda(k_p)$. By the same
arguments as before we deduce that the class $v_J=v_{i_1}v_{i_2}$
is written as a positive linear combination of the classes
$v_tv_{i_2}$, $t\in \{k_1,\ldots,k_p\}$. The extremality of the
wall $J$ implies that all these classes (there are at least two of
them by assumption) are proportional to the class $v_J$. Again,
this leads to contradiction: $v_{k_1}v_{i_2}v_i=0$ since
$\{k_1,i_2,i\}\notin K$, but $v_{i_1}v_{i_2}v_i\neq 0$ since
$\{i_1,i_2,i\}\in K$.

We proved that there are no more than four maximal cones
containing $\lambda(i_2)$. There can't be three maximal cones by
obvious geometrical reasons: the vectors
$\lambda(i_2),\lambda(i),\lambda(i')$ belong to a 2-plane and
therefore cannot form a maximal cone.

It was shown that in the sphere $K$ there exists a vertex having
either 3 or 4 neighbors. This means that in the dual polytope $P$
there exists either a triangular or quadrangular face.
\end{proof}

\begin{rem}
The existence of a strictly convex effective cone, and as a
corollary, extremal classes, is the fact, which marks out
projective smooth toric varieties among all quasitoric manifolds.
For general quasitoric manifolds we may still define the
cohomology classes $v_J\in H^{2n-2}(X;\Ro)$ corresponding to the
walls, however their nonnegative linear combinations may span the
whole space $H^{2n-2}(X;\Ro)$ rather than a strictly convex cone.
This is why it is impossible to find ``extremal'' classes with
nice properties.
\end{rem}

\end{document}